\documentclass[11pt,a4paper,english,reqno]{amsart}
\usepackage{ae,aecompl}
\usepackage[T1]{fontenc}
\usepackage[latin9]{inputenc}
\usepackage{color}
\usepackage{babel}
\usepackage{amsbsy}
\usepackage{amstext}
\usepackage{amsthm}
\usepackage{amssymb}
\usepackage{setspace}
\usepackage[bookmarks=true,bookmarksnumbered=false,bookmarksopen=false,
 breaklinks=false,pdfborder={0 0 1},backref=false,colorlinks=true]
 {hyperref}
\hypersetup{
 allcolors=blue,pdfstartview=FitH}

\makeatletter

\pdfpageheight\paperheight
\pdfpagewidth\paperwidth

\numberwithin{equation}{section}
\theoremstyle{plain}
\newtheorem{thm}{\protect\theoremname}[section]
\theoremstyle{plain}
\newtheorem{lem}[thm]{\protect\lemmaname}
\theoremstyle{plain}
\newtheorem{lyxalgorithm}[thm]{\protect\algorithmname}
\theoremstyle{plain}
\newtheorem{cor}[thm]{\protect\corollaryname}

\usepackage{enumitem}
\setlist{leftmargin=*,font=\normalfont}
\makeatother

\providecommand{\algorithmname}{Algorithm}
\providecommand{\corollaryname}{Corollary}
\providecommand{\lemmaname}{Lemma}
\providecommand{\theoremname}{Theorem}

\begin{document}
\title{On the noninjectivity of Jordan's totient functions}
\author{Hang Fu}
\address{Department of Mathematics and Computer Science, University of Basel,
Spiegelgasse 1, 4051 Basel, Switzerland}
\email{drfuhang@gmail.com}
\urladdr{https://sites.google.com/view/hangfu}
\date{\today}
\begin{abstract}
For any positive integer $k$, let $J_{k}$ be Jordan's totient function.
We show that $J_{k}$ is not injective when $k=3,4,6$.
\end{abstract}

\subjclass[2020]{11Y70}
\keywords{Jordan's totient function}
\thanks{The author acknowledges support by the Swiss National Science Foundation
Grant ``Rational points, arithmetic dynamics, and heights'' no. 219397.}
\maketitle

\section{Introduction}

For any positive integer $k$, let
\begin{equation}
J_{k}(n)=n^{k}\prod_{\substack{p\mid n\\
p\text{ prime}
}
}(1-p^{-k})=\left(\frac{n}{\text{rad}(n)}\right)^{k}\prod_{\substack{p\mid n\\
p\text{ prime}
}
}(p^{k}-1)\label{eq:1.1}
\end{equation}
be Jordan's totient function, where
\[
\text{rad}(n)=\prod_{\substack{p\mid n\\
p\text{ prime}
}
}p
\]
is the radical of $n$. This is a generalization of Euler's totient
function, which is the same as $J_{1}(n)$.

In this article, we study the noninjectivity of $J_{k}$ by solving
the equation
\begin{equation}
J_{k}(m)=J_{k}(n).\label{eq:1.2}
\end{equation}
By a solution of \eqref{eq:1.2}, we mean an unordered pair $\{m,n\}$
of distinct positive integers satisfying \eqref{eq:1.2}. Such a solution
is called primitive if $\{m/r,n/r\}$ is not a solution for any integer
$r>1$.

Clearly, $J_{1}$ and $J_{2}$ are not injective since $J_{1}(1)=J_{1}(2)$
and $J_{2}(5)=J_{2}(6)$. In \cite{MR3536148}, Bogomolov and the
author found that $J_{3}$ is not injective by giving the solution
\[
J_{3}(2^{2}\cdot37\cdot191)=J_{3}(28268)=J_{3}(28710)=J_{3}(2\cdot3^{2}\cdot5\cdot11\cdot29).
\]
This is the smallest solution of $J_{3}(m)=J_{3}(n)$ in the sense
that no solution has a smaller value of either $\max\{m,n\}$ or $\min\{m,n\}$.
However, we will see that there is a solution with a smaller value
of $P^{+}(mn)$, the largest prime factor of $mn$.

Motivated by this result, we study the same problem for other $k$.
The main theorem of this article is as follows.
\begin{thm}
\label{thm:1.1} When $k=3,4,6$, Jordan's totient function $J_{k}$
is not injective. Moreover:
\begin{enumerate}
\item For $k=3$, we have
\[
J_{3}(2\cdot29\cdot37\cdot67\cdot163)=J_{3}(2^{2}\cdot3^{3}\cdot7^{2}\cdot67^{2}).
\]
This is the only primitive solution such that $P^{+}(mn)\leq163$.
\item For $k=4$, we have
\begin{align*}
J_{4}(3^{2}\cdot13\cdot23\cdot41\cdot73\cdot193\cdot239) & =J_{4}(2^{3}\cdot5\cdot17\cdot31\cdot83\cdot307\cdot701),\\
J_{4}(3\cdot5^{2}\cdot23\cdot41\cdot73\cdot239\cdot701) & =J_{4}(2\cdot13^{2}\cdot17\cdot31\cdot83\cdot193\cdot307).
\end{align*}
These are the only primitive solutions such that $P^{+}(mn)\leq701$.
\item Any solution of $J_{6}(m)=J_{6}(n)$ satisfies $P^{+}(mn)>7.4\cdot10^{6}$.
\item There is a solution of $J_{6}(m)=J_{6}(n)$ with $P^{+}(mn)<5.2\cdot10^{9}$.
\end{enumerate}
\end{thm}

The proof relies on computer calculations. The bounds in parts (3)
and (4) could be slightly improved by using more computational resources
or by analyzing the computational results more carefully. However,
such further effort would not help us find an explicit solution, so
we leave the theorem as stated. All computations were carried out
using \texttt{Mathematica}. The code is available with the \texttt{arXiv}
version of this paper.

In Section \ref{sec:2}, we develop the lemmas and algorithms used
to find primitive solutions. In Section \ref{sec:3}, we briefly discuss
the well-known cases $k=1,2$. The purpose of this section is to illustrate
the theoretical difficulties and to justify why a computational approach
is used. Readers interested primarily in the proof of Theorem \ref{thm:1.1}
may skip this section. In Section \ref{sec:4}, we prove Theorem \ref{thm:1.1}
using computer calculations and a counting argument.

Throughout the paper, the letters $p,q,\ell$ denote primes.

\section{\label{sec:2} Some lemmas and algorithms}

Assume that $\{m,n\}$ is a solution of \eqref{eq:1.2}. Let
\begin{equation}
S=\{p:p\mid m,p\nmid n\}\qquad\text{and}\qquad T=\{p:p\mid n,p\nmid m\}.\label{eq:2.1}
\end{equation}
Then we have $S\cap T=\varnothing$, $S\cup T\neq\varnothing$, and
\begin{equation}
{\textstyle \prod_{p\in S}}(p^{k}-1)/{\textstyle \prod_{p\in T}}(p^{k}-1)\in(\mathbb{Q}^{\times})^{k}.\label{eq:2.2}
\end{equation}

Given a positive integer $C$, in order to find all primitive solutions
of \eqref{eq:1.2} with $P^{+}(mn)\leq C$, we first determine all
$S,T$ satisfying $\max(S\cup T)\leq C$ and \eqref{eq:2.2}. The
following lemma and algorithm simplify this computation.
\begin{lem}
\label{lem:2.1} Let $k=3,4,6$, let $S,T$ be sets of primes such
that $S\cap T=\varnothing$, $S\cup T\neq\varnothing$, and \eqref{eq:2.2}
holds, and let $C\geq\max(S\cup T)$ be an integer. Then for any $p\in S\cup T$,
we have
\[
P^{+}(p^{k}-1)<2C.
\]
\end{lem}

\begin{proof}
Suppose that there exist $p\in S\cup T$ and $\ell>2C$ such that
$\ell\mid p^{k}-1$. We claim that there does not exist another $q\in S\cup T$
such that $\ell\mid q^{k}-1$. Consider the roots of $x^{k}\equiv1\bmod\ell$.

Case 1: $k=3$. Since $\ell\mid p^{3}-1$ and $\ell>p-1$, we have
$\ell\mid p^{2}+p+1$. The equation $x^{3}\equiv1\bmod\ell$ has three
distinct roots
\[
1<p<\ell-p-1
\]
unless $p=C$ and $\ell=2C+1$. But then we have
\[
C(C-1)\equiv C^{2}+C+1\equiv0\bmod(2C+1),
\]
a contradiction. We note that $\ell-p-1>C$.

Case 2: $k=4$. Since $\ell\mid p^{4}-1$ and $\ell>p+1$, we have
$\ell\mid p^{2}+1$. The equation $x^{4}\equiv1\bmod\ell$ has four
distinct roots
\[
1<p<\ell-p<\ell-1.
\]
We note that $\ell-p>C$.

Case 3: $k=6$. Since $\ell\mid p^{6}-1$ and $\ell>p+1$, we have
$\ell\mid p^{2}+p+1$ or $\ell\mid p^{2}-p+1$.

(1) If $\ell\mid p^{2}+p+1$, then $x^{6}\equiv1\bmod\ell$ has six
distinct roots
\[
1<p<p+1<\ell-p-1<\ell-p<\ell-1
\]
unless $p=C$ and $\ell=2C+1$. But then we get the same contradiction
as in Case 1. We note that $\ell-p-1>C$. If $p+1\in S\cup T$, then
we have $p=2$, $C=3$, $\ell=7$, and $S\cup T=\{2,3\}$. But from
\[
2^{6}-1=3^{2}\cdot7\qquad\text{and}\qquad3^{6}-1=2^{3}\cdot7\cdot13,
\]
we cannot get \eqref{eq:2.2}, a contradiction.

(2) If $\ell\mid p^{2}-p+1$, then $x^{6}\equiv1\bmod\ell$ has six
distinct roots
\[
1<p-1<p<\ell-p<\ell-p+1<\ell-1
\]
unless $p=2$ and $\ell=3$. But then we have $\ell<2C$, a contradiction.
We note that $\ell-p>C$. If $p-1\in S\cup T$, then we have $p=3$,
$C=3$, $\ell=7$, and $S\cup T=\{2,3\}$. But then we get the same
contradiction as in Case 3 (1).

Finally, we note that since $\ell^{2}>p^{2}+p+1$, we have $v_{\ell}(p^{k}-1)=1$.
The claim implies that $v_{\ell}(q^{k}-1)=0$ for any other $q\in S\cup T$.
Therefore, \eqref{eq:2.2} cannot hold.
\end{proof}
\begin{lyxalgorithm}
\label{alg:2.2} Let $k\geq1$, let $C$ be a positive integer, and
let
\[
\mathcal{A}_{0}=\begin{cases}
\{p:p\leq C\text{ and }P^{+}(p^{k}-1)<2C\}, & \text{if }k=3,4,6,\\
\{p:p\leq C\}, & \text{otherwise}.
\end{cases}
\]
Then we do the following operation recursively: let
\begin{align*}
\mathcal{A}_{n+1}= & \;\mathcal{A}_{n}\backslash\{p\in\mathcal{A}_{n}:\text{there exists }\ell\text{ such that }v_{\ell}(p^{k}-1)>0,\\
 & \;k\nmid v_{\ell}(p^{k}-1),\text{and }v_{\ell}(q^{k}-1)=0\text{ for any }q\in\mathcal{A}_{n}\backslash\{p\}\}.
\end{align*}
The sequence $\mathcal{A}_{n}$ will stabilize after finitely many
steps.
\end{lyxalgorithm}

If $S,T$ satisfy $\max(S\cup T)\leq C$ and \eqref{eq:2.2}, then
we have $S\cup T\subseteq\mathcal{A}_{n}$ for any $n\geq0$.
\begin{lem}
\label{lem:2.3} Let $k\ge1$, let $S,T$ be sets of primes such that
$S\cap T=\varnothing$, $S\cup T\neq\varnothing$, and \eqref{eq:2.2}
holds, and let $a,b$ be the unique coprime positive integers such
that
\begin{equation}
a^{k}\prod_{p\in S}(p^{k}-1)=b^{k}\prod_{p\in T}(p^{k}-1).\label{eq:2.3}
\end{equation}
Then $S,T$ arise from a solution of \eqref{eq:1.2} if and only if
\begin{equation}
S\cap\{p:p\mid b\}=T\cap\{p:p\mid a\}=\varnothing.\label{eq:2.4}
\end{equation}
If \eqref{eq:2.4} holds, let
\begin{equation}
R=\{p:p\mid ab\}\backslash(S\cup T)\label{eq:2.5}
\end{equation}
and
\begin{equation}
m=a\prod_{p\in S\sqcup R}p\qquad\text{and}\qquad n=b\prod_{p\in T\sqcup R}p.\label{eq:2.6}
\end{equation}
Then $\{m,n\}$ is the unique primitive solution of \eqref{eq:1.2}
with associated sets $S,T$.
\end{lem}

\begin{proof}
Suppose first that $\{m',n'\}$ is a solution with associated sets
$S,T$. By \eqref{eq:1.1} and \eqref{eq:2.1}, we have
\[
\left(\frac{m'}{\text{rad}(m')}\right)^{k}\prod_{p\in S}(p^{k}-1)=\left(\frac{n'}{\text{rad}(n')}\right)^{k}\prod_{p\in T}(p^{k}-1).
\]
Comparison with \eqref{eq:2.3} gives
\begin{equation}
m'=\text{rad}(m')ad\qquad\text{and}\qquad n'=\text{rad}(n')bd\label{eq:2.7}
\end{equation}
for some positive integer $d$. By \eqref{eq:2.1}, if $p\in S$,
then $p\nmid n'$, so $p\nmid b$; similarly, if $p\in T$, then $p\nmid m'$,
so $p\nmid a$. Thus \eqref{eq:2.4} is necessary.

Conversely, assume \eqref{eq:2.4} and define $m,n$ by \eqref{eq:2.6}.
Then
\[
\text{rad}(m)=\prod_{p\in S\sqcup R}p\qquad\text{and}\qquad\text{rad}(n)=\prod_{p\in T\sqcup R}p.
\]
Therefore, by \eqref{eq:1.1} and \eqref{eq:2.3},
\[
J_{k}(m)=a^{k}\prod_{p\in S\sqcup R}(p^{k}-1)=b^{k}\prod_{p\in T\sqcup R}(p^{k}-1)=J_{k}(n),
\]
and the associated sets are exactly $S,T$.

Now we prove primitiveness. We first note that
\[
\gcd(m,n)=\prod_{p\in R}p.
\]
Suppose that a positive integer $r$ divides $\gcd(m,n)$ and $J_{k}(m/r)=J_{k}(n/r)$.
Let
\[
\alpha=\gcd(r,a)\qquad\text{and}\qquad\beta=\gcd(r,b).
\]
Since $r$ is squarefree, and each prime factor of $r$ divides exactly
one of $\alpha,\beta$, we have $r=\alpha\beta$ and $\gcd(\alpha,\beta)=1$.
By \eqref{eq:2.6}, if $p\mid\alpha$, then $v_{p}(m)\geq2$ and $v_{p}(n)=1$;
similarly, if $p\mid\beta$, then $v_{p}(n)\geq2$ and $v_{p}(m)=1$.
By \eqref{eq:1.1}, we have
\[
J_{k}(m)=\alpha^{k}J_{k}(\beta)J_{k}(m/r)\qquad\text{and}\qquad J_{k}(n)=\beta^{k}J_{k}(\alpha)J_{k}(n/r).
\]
The assumed equalities imply
\[
\alpha^{k}J_{k}(\beta)=\frac{J_{k}(m)}{J_{k}(m/r)}=\frac{J_{k}(n)}{J_{k}(n/r)}=\beta^{k}J_{k}(\alpha).
\]
Since $\gcd(\alpha,\beta)=1$, we have
\begin{align*}
\alpha^{k}\mid J_{k}(\alpha) & \Rightarrow J_{k}(\alpha)\geq\alpha^{k}\Rightarrow\alpha=1,\\
\beta^{k}\mid J_{k}(\beta) & \Rightarrow J_{k}(\beta)\geq\beta^{k}\Rightarrow\beta=1.
\end{align*}
Therefore, $r=\alpha\beta=1$, and $\{m,n\}$ is primitive.

It remains to prove uniqueness. Let $\{m',n'\}$ be a solution with
associated sets $S,T$ and let $R'=\{p:p\mid\gcd(m',n')\}$. Then
\begin{equation}
\text{rad}(m')=\prod_{p\in S\sqcup R'}p\qquad\text{and}\qquad\text{rad}(n')=\prod_{p\in T\sqcup R'}p.\label{eq:2.8}
\end{equation}
We have $R\subseteq R'$. Indeed, if $p\in R$, then $p\mid ab$ by
\eqref{eq:2.5}, so $p\mid m'n'$ by \eqref{eq:2.7}. Since $p\notin S\cup T$
by \eqref{eq:2.5}, we have $p\in R'$. By \eqref{eq:2.6}, \eqref{eq:2.7},
and \eqref{eq:2.8}, we have
\[
m'=m\left(d\prod_{p\in R'\backslash R}p\right)\qquad\text{and}\qquad n'=n\left(d\prod_{p\in R'\backslash R}p\right).
\]
Therefore, if $\{m',n'\}$ is primitive, then $\{m',n'\}=\{m,n\}$.
\end{proof}
\begin{lyxalgorithm}
\label{alg:2.4} Let $k\geq1$, let $C$ be a positive integer, let
$\mathcal{A}$ be the stabilized set given by Algorithm \ref{alg:2.2},
let
\[
\mathcal{B}=\{\ell:\ell\mid p^{k}-1,p\in\mathcal{A}\},
\]
and let $M$ be the $\#\mathcal{B}\times\#\mathcal{A}$ matrix given
by
\[
M=\left(v_{\ell}(p^{k}-1)\right)_{\ell\in\mathcal{B},p\in\mathcal{A}}.
\]
If a nonzero vector $\boldsymbol{x}=(x_{p})_{p\in\mathcal{A}}\in\{-1,0,1\}^{\mathcal{A}}$
satisfies the linear equation
\begin{equation}
M\boldsymbol{x}\equiv\boldsymbol{0}\bmod k,\label{eq:2.9}
\end{equation}
then we define
\[
S=\{p:x_{p}=1\}\qquad\text{and}\qquad T=\{p:x_{p}=-1\}.
\]
By this construction, $S,T$ satisfy \eqref{eq:2.2}. Define $a,b$
by \eqref{eq:2.3}. If $S,T$ satisfy \eqref{eq:2.4}, then we define
$R\subseteq\mathcal{B}$ by \eqref{eq:2.5} and define $m,n$ by \eqref{eq:2.6}.
If $r\leq C$ for every $r\in R$, then $\{m,n\}$ is a primitive
solution of \eqref{eq:1.2} with $P^{+}(mn)\leq C$.
\end{lyxalgorithm}

\section{\label{sec:3} The cases $k=1,2$}

For $k=1,2$, the existence of infinitely many primitive solutions
of \eqref{eq:1.2} follows from earlier work; see, for example, \cite{MR1715326}
for $k=1$, and \cite{MR2075645,MR2999152,MR3034320} for $k=2$.
We include a proof using the results of Section \ref{sec:2} to make
the primitiveness explicit and to illustrate the difficulties in extending
the argument to $k\geq3$.
\begin{thm}
[{\cite[Theorem 1.2]{MR1830570}}] \label{thm:3.1} Let $F$ be
a nonconstant polynomial with integer coefficients, let $d$ be the
largest degree of its irreducible factors, and let $r$ be the number
of distinct irreducible factors of degree $d$. Assume that $F(0)\neq0$
if $d=1$. Then for every $\epsilon>0$, the estimate
\[
\#\{p\leq x:P^{+}(F(p))\leq y\}\asymp\frac{x}{\log x}
\]
holds for all sufficiently large $x$, provided $y\geq x^{d-1/(2r)+\epsilon}$.
\end{thm}

\begin{cor}
\label{cor:3.2} When $k=1,2$, there are infinitely many primitive
solutions of \eqref{eq:1.2}.
\end{cor}

\begin{proof}
For $k=1$, take any $p$ and let
\[
S=\{p\},T=\varnothing,a=1,b=p-1.
\]
Then Lemma \ref{lem:2.3} gives a primitive solution with $P^{+}(mn)=p$.
In particular, distinct values of $p$ give distinct solutions.

For $k=2$, take $F(X)=X^{2}-1$. Then $d-1/(2r)=3/4$. Fix $0<\epsilon<1/4$
and let $y=x^{3/4+\epsilon}$. Let
\[
\mathcal{A}=\{p:y<p\leq x,P^{+}(p^{2}-1)\leq y\}\qquad\text{and}\qquad\mathcal{B}=\{\ell:\ell\leq y\}.
\]
Since $3/4+\epsilon<1$, we have $\#\mathcal{B}\leq y=o(x/\log x)$.
Then Theorem \ref{thm:3.1} implies that $\#\mathcal{A}>\#\mathcal{B}$
for all sufficiently large $x$. Thus the vectors
\[
(v_{\ell}(p^{2}-1)\bmod2)_{\ell\in\mathcal{B}}\text{ for }p\in\mathcal{A}
\]
are linearly dependent. Therefore, there is a nonempty set $S\subseteq\mathcal{A}$
such that
\[
\prod_{p\in S}(p^{2}-1)=b^{2}
\]
for some positive integer $b$. This is the same as \eqref{eq:2.3}
with $T=\varnothing$ and $a=1$. Since $\mathcal{A}\cap\mathcal{B}=\varnothing$,
the condition \eqref{eq:2.4} is satisfied. Then Lemma \ref{lem:2.3}
gives a primitive solution with $P^{+}(mn)>y$. By letting $x\to\infty$,
we obtain infinitely many distinct solutions.
\end{proof}
The key point in the proof of Corollary \ref{cor:3.2} for $k=2$
is $d-1/(2r)<1$. For $k\geq3$, the polynomial $X^{k}-1$ satisfies
$d-1/(2r)>1$, so Theorem \ref{thm:3.1} does not provide the smoothness
estimate needed to extend the argument. We also note that for $k=3,4,6$,
Lemmas \ref{lem:2.1} and \ref{lem:2.3} imply that if \eqref{eq:1.2}
has infinitely many primitive solutions, then there are infinitely
many primes $p$ such that $P^{+}(p^{k}-1)<2p$. To our knowledge,
this latter infinitude remains unproved.

\section{\label{sec:4} Proof of Theorem \ref{thm:1.1}}
\begin{proof}
[Proof of Theorem \ref{thm:1.1}] We use Algorithm \ref{alg:2.4}
to prove parts (1), (2), (3). By a solution of \eqref{eq:2.9}, we
mean a vector $\boldsymbol{x}\in\{-1,0,1\}^{\mathcal{A}}$ satisfying
\eqref{eq:2.9}.

(1) Let $k=3$ and $C=163$. Then
\begin{align*}
\mathcal{A} & =\{2,3,7,29,37,163\},\\
\mathcal{B} & =\{2,3,7,13,19,67\}.
\end{align*}
Since $\mathbb{Z}/3\mathbb{Z}$ is a field, we solve \eqref{eq:2.9}
by standard linear algebra. Up to sign, \eqref{eq:2.9} has one nonzero
solution, which yields the unique primitive solution of \eqref{eq:1.2}
with $P^{+}(mn)\leq C$.

(2) Let $k=4$ and $C=701$. Then
\begin{align*}
\mathcal{A} & =\{2,3,5,7,13,17,23,31,41,43,73,83,193,239,307,463,701\},\\
\mathcal{B} & =\{2,3,5,7,11,13,17,29,37,41,53,97,149\}.
\end{align*}
Since $\mathbb{Z}/4\mathbb{Z}$ is not a field, we first solve $M\boldsymbol{y}\equiv\boldsymbol{0}\bmod2$,
where $\boldsymbol{y}\in\{0,1\}^{\mathcal{A}}$, and then for each
solution $\boldsymbol{y}$, we decide whether it can be lifted to
a mod $4$ solution. Let $I=\{p\in\mathcal{A}:y_{p}=1\}$ be the support
of $\boldsymbol{y}$ and let
\[
N=(M_{\ell,p})_{\ell\in\mathcal{B},p\in I}
\]
be the corresponding submatrix of $M$. Suppose that $\boldsymbol{x}$
is a lift of $\boldsymbol{y}$. Then $N\boldsymbol{x}_{I}\equiv\boldsymbol{0}\bmod4$.
Let $\boldsymbol{z}=(\boldsymbol{x}_{I}+\boldsymbol{1})/2$. Then
$\boldsymbol{z}\in\{0,1\}^{I}$ and
\[
N\boldsymbol{z}\equiv\frac{N\boldsymbol{1}}{2}\bmod2.
\]
Note that $N\boldsymbol{1}$ has even entries, so the right-hand side
is well-defined. From each solution $\boldsymbol{z}$, we recover
$\boldsymbol{x}$ by
\[
\boldsymbol{x}_{I}=2\boldsymbol{z}-\boldsymbol{1}\qquad\text{and}\qquad x_{p}=0\text{ for }p\notin I.
\]
Up to sign, \eqref{eq:2.9} has four nonzero solutions. Two of them
satisfy \eqref{eq:2.4} and yield the only two primitive solutions
of \eqref{eq:1.2} with $P^{+}(mn)\leq C$.

(3) Let $k=6$ and $C=7.4\cdot10^{6}$. Then $\#\mathcal{A}=3460$
and $\#\mathcal{B}=3443$. We check that
\[
\dim\ker_{\mathbb{Z}/2\mathbb{Z}}M=\dim\ker_{\mathbb{Z}/3\mathbb{Z}}M=17.
\]
The equation $M\boldsymbol{y}\equiv\boldsymbol{0}\bmod2$ has $2^{17}-1$
nonzero solutions. For each solution $\boldsymbol{y}$, we decide
whether it can be lifted to a mod $6$ solution. Let $\{\boldsymbol{z}_{i}\}_{i=1}^{17}$
be a basis of $\ker_{\mathbb{Z}/3\mathbb{Z}}M$ and let $I=\{p\in\mathcal{A}:y_{p}=0\}$.
Suppose that $\boldsymbol{x}$ is a lift of $\boldsymbol{y}$. Then
there is a nonzero vector $(c_{i})_{i=1}^{17}$ such that
\[
\bar{\boldsymbol{x}}=\boldsymbol{x}\bmod3=\sum_{i=1}^{17}c_{i}\boldsymbol{z}_{i}\qquad\text{and}\qquad\boldsymbol{0}=\bar{\boldsymbol{x}}_{I}=\sum_{i=1}^{17}c_{i}(\boldsymbol{z}_{i})_{I},
\]
which implies that $\{(\boldsymbol{z}_{i})_{I}\}_{i=1}^{17}$ is linearly
dependent. However, we check that no nonzero solution $\boldsymbol{y}$
satisfies this necessary condition. Therefore, \eqref{eq:2.9} has
no nonzero solution, and any solution of \eqref{eq:1.2} must satisfy
$P^{+}(mn)>C$.

(4) Let $k=6$ and $C=2.65\cdot10^{9}$. Then
\[
\#\mathcal{A}=3080877\qquad\text{and}\qquad\#\mathcal{B}=1726159.
\]
For this existence proof, we need a variant of Algorithm \ref{alg:2.2}.
Let $\mathcal{C}_{0}=\mathcal{A}\backslash\mathcal{B}$. Then we do
the following operation recursively: let
\begin{align*}
\mathcal{C}_{n+1}= & \;\mathcal{C}_{n}\backslash\{p\in\mathcal{C}_{n}:\text{there exists }\ell\text{ such that }v_{\ell}(p^{6}-1)>0,\\
 & \;6\nmid v_{\ell}(p^{6}-1),\text{and }v_{\ell}(q^{6}-1)=0\text{ for any }q\in\mathcal{C}_{n}\backslash\{p\}\}.
\end{align*}
The sequence $\mathcal{C}_{n}$ will stabilize to a set $\mathcal{C}$
after finitely many steps. Let
\[
\mathcal{D}=\{\ell:\ell\mid p^{6}-1,p\in\mathcal{C}\}.
\]
Then
\begin{equation}
\#\mathcal{C}=2752446\qquad\text{and}\qquad\#\mathcal{D}=1610722.\label{eq:4.1}
\end{equation}
For each $\ell\in\mathcal{D}$, let
\[
\mathcal{C}_{\ell}=\{p\in\mathcal{C}:\ell\mid p^{6}-1\}.
\]
We check that $\#\mathcal{C}_{\ell}\geq2$ for every $\ell\in\mathcal{D}$.
Let
\begin{align*}
\mathcal{D}_{i} & =\{\ell\in\mathcal{D}:\#\mathcal{C}_{\ell}=i\}\text{ for }i=2,3,4,\\
\mathcal{D}_{5} & =\{\ell\in\mathcal{D}:\#\mathcal{C}_{\ell}\geq5\}.
\end{align*}
Then
\begin{equation}
\#\mathcal{D}_{2}=767828,\#\mathcal{D}_{3}=271105,\#\mathcal{D}_{4}=138124,\#\mathcal{D}_{5}=433665.\label{eq:4.2}
\end{equation}
Let $G$ be the graph with vertex set $\mathcal{C}$ and edge set
$\mathcal{E}=\{\mathcal{C}_{\ell}:\ell\in\mathcal{D}_{2}\}$. Let
$H$ be a connected component of $G$ with $V(H)$ vertices and $E(H)$
edges. Then
\[
E(H)\geq\begin{cases}
V(H)-1, & \text{if }H\text{ is bipartite},\\
V(H), & \text{if }H\text{ is non-bipartite}.
\end{cases}
\]
The second inequality is because any non-bipartite component must
contain a cycle. Summing over $H$ gives
\[
\#\mathcal{C}-r\leq\#\mathcal{E}\leq\#\mathcal{D}_{2}\qquad\text{and}\qquad r\geq\#\mathcal{C}-\#\mathcal{D}_{2},
\]
where
\[
r=\#\{\text{bipartite connected components of }G\}.
\]
Let $\mathcal{F}$ consist of subsets $U\subseteq\mathcal{C}$ obtained
as follows:
\begin{itemize}
\item choose exactly one color class from every bipartite component;
\item choose no vertices from every non-bipartite component.
\end{itemize}
Thus
\begin{equation}
\#\mathcal{F}=2^{r}\geq2^{\#\mathcal{C}-\#\mathcal{D}_{2}}.\label{eq:4.3}
\end{equation}
Consider the map
\begin{align*}
\Phi:\mathcal{F} & \to(\mathbb{Z}/6\mathbb{Z})^{\mathcal{D}},\\
U & \mapsto\left({\textstyle \sum_{p\in U}}v_{\ell}(p^{6}-1)\bmod6\right)_{\ell\in\mathcal{D}}.
\end{align*}
We check that $v_{\ell}(p^{6}-1)=1$ for every $\ell\in\mathcal{D}_{2}\cup\mathcal{D}_{3}\cup\mathcal{D}_{4}$
and every $p\in\mathcal{C}_{\ell}$. As a result, we have
\[
\Phi(U)_{\ell}\in\begin{cases}
\{0,1,2,3\}, & \text{if }\ell\in\mathcal{D}_{3},\\
\{0,1,2,3,4\}, & \text{if }\ell\in\mathcal{D}_{4},\\
\{0,1,2,3,4,5\}, & \text{if }\ell\in\mathcal{D}_{5},
\end{cases}
\]
and
\[
\Phi(U)_{\ell}=\begin{cases}
1, & \text{if }\ell\in\mathcal{D}_{2}\text{ and }\text{Comp}(\mathcal{C}_{\ell})\text{ is bipartite},\\
0, & \text{if }\ell\in\mathcal{D}_{2}\text{ and }\text{Comp}(\mathcal{C}_{\ell})\text{ is non-bipartite},
\end{cases}
\]
where $\text{Comp}(\mathcal{C}_{\ell})$ is the component containing
$\mathcal{C}_{\ell}$. In particular, if $\ell\in\mathcal{D}_{2}$,
then the value of $\Phi(U)_{\ell}$ does not depend on $U$. Therefore,
we have
\begin{equation}
\#\Phi(\mathcal{F})\leq4^{\#\mathcal{D}_{3}}5^{\#\mathcal{D}_{4}}6^{\#\mathcal{D}_{5}}.\label{eq:4.4}
\end{equation}
Substituting \eqref{eq:4.1} and \eqref{eq:4.2} into \eqref{eq:4.3}
and \eqref{eq:4.4}, we find that
\begin{equation}
\#\mathcal{F}\geq2^{\#\mathcal{C}-\#\mathcal{D}_{2}}>4^{\#\mathcal{D}_{3}}5^{\#\mathcal{D}_{4}}6^{\#\mathcal{D}_{5}}\geq\#\Phi(\mathcal{F}).\label{eq:4.5}
\end{equation}
By the pigeonhole principle, there are distinct $S',T'\in\mathcal{F}$
such that $\Phi(S')=\Phi(T')$. Then $S=S'\backslash T'$ and $T=T'\backslash S'$
satisfy $S\cap T=\varnothing$, $S\cup T\neq\varnothing$, and \eqref{eq:2.2}.
Define $a,b,R,m,n$ as in Lemma \ref{lem:2.3}. Since $\mathcal{C}\cap\mathcal{D}=\varnothing$,
the condition \eqref{eq:2.4} is satisfied. Therefore, $\{m,n\}$
is a solution of \eqref{eq:1.2}. Finally, we check that $P^{+}(mn)\leq\max(\mathcal{C}\cup\mathcal{D})<5.2\cdot10^{9}$.
\end{proof}
We make some remarks on the computation in the proof of Theorem \ref{thm:1.1}.
For parts (1) and (2), Lemma \ref{lem:2.1} is not really necessary
because $C$ is small. However, for parts (3) and (4), Lemma \ref{lem:2.1}
indeed simplifies the computation greatly. In part (3), we need to
check $2^{\dim\ker_{\mathbb{Z}/2\mathbb{Z}}M}-1$ cases. Computational
evidence shows that $\dim\ker_{\mathbb{Z}/2\mathbb{Z}}M$ increases
quickly as $C$ increases. This is why we are unable to significantly
improve the lower bound. In part (4), we use the inequality \eqref{eq:4.5}
without knowing the structure of $G$. Studying the details of $G$
could improve the upper bound. We do not pursue this here.

\bibliographystyle{amsalpha}
\bibliography{Jordan}

\end{document}